\documentclass[11pt,a4paper]{amsart}
\usepackage{amsmath,amsthm,amssymb,latexsym}
\usepackage{graphicx}

\newcommand{\bn}{{\mathbb{N}}}
\newcommand{\br}{{\mathbb{R}}}
\newcommand{\bp}{{\mathbb{P}}}

\newcommand{\bc}{{\mathbb{C}}}

\newcommand{\bq}{{\mathbb{Q}}}

\newcommand{\ch}{{\mathcal{H}}}

\newcommand{\ce}{{\mathcal{E}}}
\newcommand{\cm}{{\mathcal{M}}}
\newcommand{\cv}{{\mathcal{V}}}

\renewcommand{\b}{\beta}
\renewcommand{\l}{\lambda}
\newcommand{\s}{\sigma}

\newcommand{\dd}{\Delta}
\renewcommand{\o}{\omega}

\renewcommand{\gg}{\Gamma}

\newcommand{\ag}{A(\gg)\,}

\newcounter{images}
\newcommand{\image}{\refstepcounter{images}%
Fig.\ \arabic{images} }

\allowdisplaybreaks
\numberwithin{equation}{section}

\newtheorem{theorem}{Theorem}[section]

\newtheorem{corollary}[theorem]{Corollary}
\newtheorem{proposition}[theorem]{Proposition}

\theoremstyle{definition}

\newtheorem{remark}[theorem]{Remark}

\newtheorem{example}[theorem]{Example}

\begin{document}

\title[Comb graphs]
{Spectra of comb graphs with tails}
\author[L. Golinskii]{L. Golinskii}

\address{B. Verkin Institute for Low Temperature Physics and Engineering of the
National Academy of Sciences of Ukraine, 47 Science ave.,  61103 Kharkiv, Ukraine}
\email{golinskii@ilt.kharkov.ua}

\date{\today}


\maketitle

\begin{abstract}
Given two graphs, a backbone and a finger, a comb product is a new graph obtained by grafting a copy
of the finger into each vertex of the backbone. We study the comb graphs in the case when both components
are the paths of order $n$ and $k$, respectively, as well as the above comb graphs with an infinite
ray attached to some of their vertices. A detailed spectral analysis is carried out in both situations.
\end{abstract}

\section{Introduction}
\label{s1}

We begin with some rudiments of the graph theory. For the sake of simplicity we restrict ourselves
with simple, connected, undirected, finite or infinite (countable) unweighted graphs, although
the method works for weighted multigraphs and graphs with loops as well.
We will enumerate the vertex set $\cv(\gg)$ with positive integers $\bn=\{1,2,\ldots\}$,
$\{v\}_{v\in \cv}=\{j\}_{j=1}^\o$, $\o\le\infty$. The symbol $i\sim j$ means that the vertices
$i$ and $j$ are incident, i.e., $\{i,j\}$ belongs to the edge set $\ce(\gg)$.

The degree (valency) of a vertex $v\in\cv(\gg)$ is a number $\nu(v)$ of edges emanating from $v$.
A graph $\gg$ is said to be locally finite, if $\nu(v)<\infty$ for all $v\in\cv(\gg)$, and
uniformly locally finite, if $\sup_{\cv}\nu(v)<\infty$.

The spectral graph theory studies spectra and spectral properties of certain
matrices related to graphs (more precisely, operators generated by such matrices in the standard
basis $\{e_k\}_{k\in\bn}$ and acting on the corresponding Hilbert spaces $\bc^n$ or
$\ell^2=\ell^2(\bn)$). One of the most notable of them is the {\it adjacency matrix} $A(\gg)$
\begin{equation}\label{adjmat}
A(\gg)=\|a_{ij}\|_{i,j=1}^\o, \quad
a_{ij}=\left\{
  \begin{array}{ll}
    1, & \{i,j\}\in\ce(\gg); \\
    0, & \hbox{otherwise.}
  \end{array}
\right.
\end{equation}
The corresponding adjacency operator (sometimes called the shift operator) will be denoted by the same symbol. It acts as
\begin{equation}\label{adjop}
A(\gg)\,e_k=\sum_{j\sim k}e_j, \qquad k=1,2,\ldots,\o.
\end{equation}
Clearly, $\ag$ is a symmetric, densely-defined linear operator, whose domain is the set of all finite
linear combinations of the basis vectors. The operator $A(\gg)$ is bounded and self-adjoint on $\ell^2=\ell^2(\bn)$
if and only if the graph $\gg$ is uniformly locally finite.

In what follows, under the {\it spectrum $\s(\gg)$ of a graph} $\gg$ we always mean the spectrum $\s(\ag)$ of its adjacency operator. The set of isolated
eigenvalues of finite multiplicity is usually referred to as the {\it discrete spectrum} $\s_d(\gg)$.

We proceed with the definition of a comb product (comb graph), see, e.g., \cite[Section 8.5]{hooba}.

Let $\gg^{(1)}=\bigl(\cv^{(1)}, \ce^{(1)}\bigr)$, $\gg^{(2)}=\bigl(\cv^{(2)}, \ce^{(2)}\bigr)$ be two graphs,
the first one is referred to as a backbone, and the second as a finger, and assume that the second graph is given a distinguished
vertex $0$. Consider the Cartesian product $\cv=\cv^{(1)}\times \cv^{(2)}$, which is a set of vertices for a new graph $\cv$. We say
that $(x_1,x_2)\sim (x_1',x_2')$, so $\{(x_1,x_2), (x_1',x_2')\}\in\ce$, if one of the following is satisfied:

\noindent
(i). $x_1\sim x_1'$, $x_2=x_2'=0$;

\noindent
(ii). $x_1=x_1'$, $x_2\sim x_2'$.

\noindent
Then $\gg=(\cv,\ce)$ becomes a graph, known as the {\it comb graph} or the {\it comb product} of $\gg^{(1)}$ and $\gg^{(2)}$ with a contact
vertex $0\in\cv^{(2)}$. The standard notation here is $\gg=\gg^{(1)}\rhd\gg^{(2)}$. The comb graph is (uniformly) locally finite
whenever so are its components $\gg^{(j)}$, $j=1,2$. In fact, the comb graph is obtained by grafting a copy of $\gg^{(2)}$ at a vertex $0$ into
each vertex of $\gg^{(1)}$. If $\gg^{(1)}=\bp_n$, the path of order $n$, the comb product $\bp_n\rhd\gg$ is given on Fig.1.

\begin{picture}(380, 120)

\multiput(42,40) (60,0) {2} {\line(1, 0) {56}}
\multiput(282,40) (60,0) {1} {\line(1, 0) {56}}
\multiput(40,40) (60,0) {3} {\circle* {4}}
\multiput(200,40) (20,0) {3} {\circle* {2}}
\multiput(280,40) (60,0) {2} {\circle* {4}}

\put(40, 60){\circle{80}}     \put(36, 56){\Large{$\Gamma$}}
\put(100, 60){\circle{80}}    \put(96, 56){\Large{$\Gamma$}}
\put(160, 60){\circle{80}}    \put(156, 56){\Large{$\Gamma$}}
\put(280, 60){\circle{80}}    \put(276, 56){\Large{$\Gamma$}}
\put(340, 60){\circle{80}}    \put(336, 56){\Large{$\Gamma$}}
\put(162, 40){\line(1,0) {20}}
\put(258, 40){\line(1,0) {20}}

\put(40, 28){$1$} \put(100, 28){$2$}  \put(160, 28){$3$}
\put(272, 28){$n-1$} \put(338, 28){$n$}

\multiput(20,60) (60,0) {3} {\circle* {4}}
\multiput(40,80) (60,0) {3} {\circle* {4}}
\multiput(260,60) (60,0) {2} {\circle* {4}}
\multiput(280,80) (60,0) {2} {\circle* {4}}
\put(-8,60){$n+1$} \put(52,60){$n+2$} \put(112,60){$n+3$}
\put(226,60){$2n-1$}  \put(304,60){$2n$}
\put(28,88){$2n+1$} \put(88,88){$2n+2$} \put(148,88){$2n+3$}
\put(268,88){$3n-1$} \put(336,88){$3n$}
\end{picture}

\begin{center}\image \end{center}

\bigskip

\begin{remark}
Although enumeration of the vertex set of a graph does not affect its spectral properties, some particular enumerations can be more convenient
for computational purposes. Let us explain this on the above comb graph.

The first way to label the vertices looks as follows. We use the numbers $\{1,2,\ldots,k\}$ to enumerate the vertices of the first copy of $\gg$,
the next $k$ numbers $\{k+1,k+2,\ldots,2k\}$ to enumerate the vertices of the second copy of $\gg$ etc. The adjacency matrix now takes the form
\begin{equation*}
A(\bp_n\rhd\gg)=
\begin{bmatrix}
 A(\gg) & E & & & \\
 E & A(\gg) & E & & \\
  & \ddots  & \ddots & \ddots &  & \\
  &   &  & A(\gg) & E & \\
  &   &  &  E  & A(\gg) &
\end{bmatrix}, \qquad
E=\|E_{ij}\|_{i,j=1}^k,
\end{equation*}
where $E_{11}=1$, and $E_{ij}=0$ for the rest of the entries. The matrix structure is transparent but not good enough for computation.

In what follows we stick to another way of enumeration, see Fig.1. Precisely, the numbers $\{1,2,\ldots,n\}$ label the vertices of the backbone $\bp_n$,
then come $n$ identical vertices of the ``first level'' labeled with $\{n+1,n+2,\ldots, 2n\}$, then $n$ identical vertices of the ``second level'' etc.
The adjacency matrix $A(\bp_n\rhd\gg)$ can be obtained by using the ``inflation'' procedure: take the adjacency matrix $A(\gg)$ and replace the first
zero entry $a_{11}=0$ with the $n\times n$ matrix, known as a {\it discrete Laplacian of order $n$},
$$ 0_{11} \ \ \longrightarrow \ \
J_n=\begin{bmatrix}
 0 & 1 & & & \\
 1 & 0 & 1 & & \\
  &  \ddots  & \ddots & \ddots &  & \\
  &   & \ddots & 0 & 1 & \\
  &   &  &  1  & 0 &
\end{bmatrix}, $$
the rest of the zero entries with $n\times n$ zero matrices $0_n$, and all unit entries $1$ with $n\times n$ unit matrices $I_n$, see \eqref{infl}.

For a formal description of the adjacency matrix of a comb product see \cite[Lemma 8.35]{hooba}.
\end{remark}

To define a main object under consideration -- a comb graph with tail -- we recall first the operation of coupling by means of the bridge, well known
for finite graphs, see, e.g., \cite[Theorem 2.12]{CDS80}.

Let $\gg_j$, $j=1,2$, be two graphs with no common vertices, with the vertex sets and edge
sets $\cv(\gg_j)$ and $\ce(\gg_j)$, respectively, and let $v_j\in \cv(\gg_j)$. A graph
$\gg=\gg_1+\gg_2$ is called a {\it coupling by means of the bridge $\{v_1,v_2\}$} if
\begin{equation}\label{defcoup}
\cv(\gg)=\cv(\gg_1)\cup \cv(\gg_2), \qquad \ce(\gg)=\ce(\gg_1)\cup \ce(\gg_2)\cup \{v_1,v_2\}.
\end{equation}
So we join $\gg_2$ to $\gg_1$ by the new edge between $v_2$ and $v_1$, see Fig.2.

\begin{picture}(300, 70)
\put(120, 30){\circle{80}}   \put(114, 24){\Large{$\Gamma_1$}}

\multiput(142, 30) (50,0) {2} {\circle* {4}}
\put(142, 30) {\line(1,0) {48}}
\put(214, 30){\circle{80}}  \put(208, 24){\Large{$\Gamma_2$}}

\put(142, 34) {$v_1$} \put(180, 34) {$v_2$}
\end{picture}

\begin{center}\image \end{center}

\bigskip

In general, the adjacency matrix $A(\gg_1+\gg_2)$ takes the form of a block operator matrix
\begin{equation}\label{adjcoup}
A(\gg)=\begin{bmatrix}
A(\gg_1) & E \\
E^* & A(\gg_2)
\end{bmatrix}, \qquad
E=\begin{bmatrix}
1 & 0 & 0 & \ldots \\
0 & 0 & 0 & \ldots \\
0 & 0 & 0 & \ldots \\
\vdots & \vdots & \vdots &
\end{bmatrix}.
\end{equation}
If the graph $\gg_1$ is finite, $\cv(\gg_1)=\{1,2,\ldots,m\}$, and $\cv(\gg_2)=\{j\}_{j=m+1}^\o$,
we can with no loss of generality put $v_1=1$, $v_2=m+1$, so the adjacency matrix $A(\gg)$ can
be written as a block operator matrix

\begin{equation}\label{adjcoupfin}
A(\gg)=\begin{bmatrix}
A(\gg_1) & E_m \\
E_m^* & A(\gg_2)
\end{bmatrix}, \qquad
E_m=\begin{bmatrix}
1 & 0 & 0 & \ldots \\
0 & 0 & 0 & \ldots \\
\vdots & \vdots & \vdots & \\
0 & 0 & 0 & \ldots &
\end{bmatrix}
\end{equation}
has $m$ rows. If $\gg_2=\bp_\infty$, the one-sided infinite path, we can view the coupling
$\gg=\gg_1+\bp_\infty$ as a finite graph with the tail. Now

\begin{equation}\label{freejac}
A(\gg_2)=
J_\infty:=
\begin{bmatrix}
 0 & 1 & 0 & 0 & \\
 1 & 0 & 1 & 0 & \\
 0 & 1 & 0 & 1 & \\
     & \ddots  & \ddots & \ddots & \ddots
\end{bmatrix}
\end{equation}
is an infinite Jacobi matrix called a {\it discrete Laplacian}.

The spectral theory of infinite graphs with one or several rays attached to certain finite graphs
was initiated in \cite{LeNi-dan, LeNi-umzh, Niz14} wherein a number of particular examples of
graphs was examined. In \cite{Gol16} this collection was enlarged considerably. The canonical form of the
adjacency matrix $A(\gg_1+\bp_\infty)$ suggested there enabled one to compute the spectrum of such graph
by using the spectral theory of Jacobi matrices of finite rank. In \cite{Gol17} we study the spectra of the graphs
with tails by means of the Schur complement.

The problem we address in the paper concerns the structure of the spectrum of the basic comb graph with the tail, see Fig.5,
\begin{equation}\label{maingra}
H_{n,k}:=\gg_{n,k}+\bp_\infty, \qquad \gg_{n,k}:=\bp_n \rhd \bp_k, \quad n,k\ge2.
\end{equation}
Here is our main result.

\begin{theorem}\label{mainthe}
For the comb graph $H_{n,k}$ the following hold.

\noindent
{\rm (i)}. \ The spectrum $\s(H_{n,k})=[-2,2]\cup\s_d(H_{n,k})$, the disjoint union, the discrete spectrum is finite and symmetric
with respect to the origin.

\noindent
{\rm (ii)}. \ Let $\s_d^+(H_{n,k}):=\s_d(H_{n,k})\cap (2,\infty)$ be the positive part of the discrete spectrum. Its cardinality equals
\begin{equation}\label{numpds}
\begin{split}
\bigl|\s_d^+(H_{n,k})\bigr| &=\Bigl[\frac{\o_k}{\pi}\,(n+1)\Bigr]+\chi(a_{n,k}-1), \\
\o_k &:=\arccos\,\frac{k+1}{2k}\,, \qquad a_{n,k}:=\frac{\sin n\o_k}{\sin (n+1)\o_k}\,,
\end{split}
\end{equation}
$\chi$ is the Heaviside function, $[x]$ is the integer part of a real number $x$.

\noindent
{\rm (iii)}.
\begin{equation}\label{inclspe}
\s_d^+(H_{n,k})\subset \Bigl[2, \frac52\Bigr], \ n,k\ge2; \quad \s_d^+(H_{n,2})\subset \bigl[2, 1+\sqrt2\bigr], \  n\ge2.
\end{equation}
Both intervals are the smallest ones with such property.

\noindent
{\rm (iv)}. \ The multiplicity of each eigenvalue $\nu\in\s_d(H_{n,k})$ is at most $2$, and the number of eigenvalues of multiplicity $2$ is
at most $4$.
\end{theorem}

Let us mention a recent manuscript \cite{abcre}, wherein some partial results concerning the spectrum of the graph $H_{n,2}$ are obtained.

We proceed as follows. In Section \ref{s2} we study the spectrum of the basic comb graph $\gg_{n,k}$ (location and upper bounds). The comb
graph with the tail $H_{n,k}$ \eqref{maingra} is examined in Section \ref{s3}, where the main result is proved. We include
some auxiliary results from linear algebra, number theory, and perturbation theory in a short Appendix at the end of the paper.

\section{Basic comb graph of height $k$}
\label{s2}

The spectral analysis of the comb graph $\gg_{n,k}$, see Fig.3, is performed in this section.
We take an endpoint of $\bp_k$ as the contact vertex $0$.

\vskip 0.8cm

\begin{picture}(360, 160)
\multiput(60,40) (30,0) {3} {\circle* {4}}
\multiput(140,40) (10,0) {3} {\circle* {2}}
\multiput(180,40) (30,0) {2} {\circle* {4}}
\multiput(60,70) (30,0) {3} {\circle* {4}}
\multiput(180,70) (30,0) {2} {\circle* {4}}
\multiput(60,130) (30,0) {3} {\circle* {4}}
\multiput(180,130) (30,0) {2} {\circle* {4}}
\multiput(60,160) (30,0) {3} {\circle* {4}}
\multiput(180,160) (30,0) {2} {\circle* {4}}
\multiput(60,90) (0,10) {3} {\circle* {2}}
\multiput(90,90) (0,10) {3} {\circle* {2}}
\multiput(120,90) (0,10) {3} {\circle* {2}}
\multiput(180,90) (0,10) {3} {\circle* {2}}
\multiput(210,90) (0,10) {3} {\circle* {2}}
\multiput(62,40) (30,0) {2} {\line(1, 0) {26}}
\multiput(182,40) (30,0) {1} {\line(1, 0) {26}}
\multiput(60,42) (30,0) {1} {\line(0, 1) {26}}
\multiput(60,132) (30,0) {1} {\line(0, 1) {26}}
\multiput(90,42) (30,0) {1} {\line(0, 1) {26}}
\multiput(90,132) (30,0) {1} {\line(0, 1) {26}}
\multiput(120,42) (30,0) {1} {\line(0, 1) {26}}
\multiput(120,132) (30,0) {1} {\line(0, 1) {26}}
\multiput(180,42) (30,0) {1} {\line(0, 1) {26}}
\multiput(180,132) (30,0) {1} {\line(0, 1) {26}}
\multiput(210,42) (30,0) {1} {\line(0, 1) {26}}
\multiput(210,132) (30,0) {1} {\line(0, 1) {26}}
\put(58,28) {$1$}
\put(88,28) {$2$}
\put(118,28) {$3$}
\put(170,28) {$n-1$}
\put(208,28) {$n$}
\put(32,68) {$n+1$}
\put(62,68) {$n+2$}
\put(214,68) {$2n$}
\put(0,158) {$(k-1)n+1$}
\put(214,158) {$kn$}

\multiput(122,40) (15,0) {1} {\line(1, 0) {11}}
\multiput(168,40) (15,0) {1} {\line(1, 0) {11}}
\put(250,44){\Large{$\Gamma_{n,k}$}}
\end{picture}

\begin{center}\image \end{center}

\bigskip

The Chebyshev polynomials of the second kind enter the expression for the characteristic polynomial and the spectrum of $\gg_{n,k}$.
They are defined by the relations
\begin{equation}\label{chedef}
U_m(\cos x)=\frac{\sin(m+1)x}{\sin x}\,, \quad U_m(t)=2^m t^m+\ldots,  \qquad m=0,1,\ldots.
\end{equation}
The corresponding monic polynomials
\begin{equation}\label{monche}
V_m(t):=U_m\Bigl(\frac{t}2\Bigr)=t^m+\ldots, \qquad m=0,1,\ldots,
\end{equation}
satisfy the three-term recurrence relation
\begin{equation}\label{3term}
\begin{split}
t\,V_{m+1}(t) &=V_{m+2}(t)+V_{m}(t),  \\
V_0(t) &=1, \quad V_1(t)=t, \quad V_2(t)=t^2-1.
\end{split}
\end{equation}
The set of the roots of $V_m$ comes from \eqref{chedef}
\begin{equation}\label{zeros}
V_m(t)=\prod_{j=1}^m \bigl(t-t_{j,m}\bigr), \qquad t_{j,m}:=2\cos\frac{j\pi}{m+1}\,.
\end{equation}
It is easy to see that $V_m(-t)=(-1)^m V_m(t)$, so the set $\{t_{j,m}\}_{j=1}^m$ is symmetric with respect to the origin.

The results in Theorem \ref{fincomb} below are not completely new. For example, formula \eqref{charpol} in (i) follows from the general formula
for the characteristic polynomial of a comb product \cite{godmckay}, see also \cite[Theorem 1]{rogujisa}. Yet, our strategy is somewhat different
and, what is more to the point, some elements of the construction are applied later on to the study of infinite graphs with tails.

\begin{theorem}\label{fincomb}
For the comb graph $\gg_{n,k}$ the following hold.

\noindent
{\rm (i)}. \ The characteristic polynomial of $\gg_{n,k}$ is
\begin{equation}\label{charpol}
P(\l, \gg_{n,k})=V_n(v_k(\l))\cdot V_{k-1}^n(\l), \ \ v_k(\l):=\frac{V_{k}(\l)}{V_{k-1}(\l)}\,, \ \ \l\not=t_{j,k-1}.
\end{equation}
The eigenvalues of $\gg_{n,k}$ are simple, and the spectrum $\s(\gg_{n,k})$ is symmetric with respect to the origin.

\noindent
{\rm (ii)}. \ Let $\s^+(\gg_{n,k}):=\s(\gg_{n,k})\cap (2,\infty)$. Its cardinality equals
\begin{equation}\label{numeig1}
\bigl|\s^+(\gg_{n,k})\bigr|=\Bigl[\frac{\o_k}{\pi}\,(n+1)\Bigr], \qquad \o_k:=\arccos\,\frac{k+1}{2k}\,.
\end{equation}

\noindent
{\rm (iii)}. \ $\s(\gg_{n,k})\subset \bigl[-\frac52, \frac52\bigr]$ for all $n,k\ge2$, and this is the smallest interval with such property.
\end{theorem}
\begin{proof}
(i). The inflation procedure described in Introduction, applied to the graph $\gg_{n,k}$, leads to the following expression as a block matrix with $k^2$ blocks,
each one belongs to the class $\cm_n$ of $n\times n$ matrices
\begin{equation}\label{infl}
\l-A(\gg_{n,k})=
\begin{bmatrix}
 \l-J_n & -I_n & & & \\
 -I_n & \l I_n & -I_n & & \\
  & \ddots  & \ddots & \ddots &  & \\
  &   &  & \l I_n & -I_n & \\
  &   &  &  -I_n  & \l I_n &
\end{bmatrix} =
\begin{bmatrix}
\l-J_n & B \\
C & D
\end{bmatrix}.
\end{equation}
Now $D=D(\l)$ is the block matrix of scalar type (see Appendix), with the symbol
$$ D=I(d), \qquad d=d(\l)=\l-J_{k-1}\in\cm_{k-1}, $$
(for $k=2$ we have $d=\l 1$). So, $|D|=|d|^n$, see \eqref{bmst}.

By the recurrence relations \eqref{3term}, the value $|d|=\det d$ agrees with the monic Chebyshev polynomial of the second kind
\begin{equation}\label{symdet}
|d(\l)|=\begin{vmatrix}
 \l & -1 & & & \\
 -1 & \l & -1 & & \\
  & \ddots  & \ddots & \ddots &  & \\
  &   &  & \l  & -1 & \\
  &   &  &  -1  & \l &
\end{vmatrix}=V_{k-1}(\l),
\end{equation}
$|D(\l)|=V_{k-1}^n(\l)$. So,  $|D(\l)|=0$ if and only if
$$ \l=t_{j,k-1}, \qquad j=1,2,\ldots,k-1. $$

Next, when $\l\not=t_{j,k-1}$, we have, by \eqref{bmst},
$$ D^{-1}=I\bigl(d^{-1}\bigr)=\|(\l-J_{k-1})^{-1}_{ij}\cdot I_n\|_{i,j=1}^{k-1}, $$
and so $BD^{-1}C=(\l-J_{k-1})^{-1}_{11}\cdot I_n$, the scalar matrix from $\cm_n$. The scalar factor
\begin{equation*}
G_{m}(\l)=(\l-J_{m})^{-1}_{11}=\frac{V_{m-1}(\l)}{V_m(\l)}
\end{equation*}
is called the Green's function. Due to the second relation in \eqref{blma1}, we have
$$ P(\l, \gg_{n,k}):=|\l-A(\gg_{n,k})|=|\bigl(\l-G_{k-1}(\l)\bigr)\,I_n-J_n|\cdot V_{k-1}^n(\l). $$
It follows from \eqref{3term} that
$$ \l-G_{k-1}(\l)=\l-\frac{V_{k-2}(\l)}{V_{k-1}(\l)}=\frac{V_{k}(\l)}{V_{k-1}(\l)}=v_k(\l)\,, $$
so \eqref{charpol} is proved.

Since $w_k(-\l)=-w_k(\l)$, we have
\begin{equation*}
P(-\l,\gg_{n,k})=(-1)^{kn}\,P(\l,\gg_{n,k}),
\end{equation*}
so $\s(\gg_{n,k})$ is symmetric with respect to the origin.

Note that each point $t_{j,k-1}$, $j=1,2,\ldots,k-1$, is regular for the right side of \eqref{charpol}, so we can pass to the
limit for $\l\to t_{j,k-1}$ to obtain
\begin{equation}\label{charpol1}
P(t_{j,k-1}, \gg_{n,k})=V_k^n(t_{j,k-1})\not=0.
\end{equation}
Thereby, $\l\in\s(\gg_{n,k})$ if and only if any one of the following $n$ relations holds
\begin{equation}\label{spefin}
V_n(v_k(\l))=0 \ \Leftrightarrow \ v_k(\l)=t_{j,n}=2\cos\frac{j\pi}{n+1}\,, \qquad j=1,\ldots,n.
\end{equation}

Since the zeros of $V_{k-1}$ and $V_k$ interlace, the function $v_k$ admits the following representation
\begin{equation}\label{inter}
v_k(t)=t+\sum_{j=1}^{k-1} \frac{d_{j,k}}{t_{j,k-1}-t}, \qquad d_{j,k}>0.
\end{equation}
It is clear from \eqref{inter} that $v_k$ is monotone increasing from $-\infty$ to $+\infty$ on each interval
\begin{equation}\label{intval}
L_{m,k}:=\bigl(t_{m,k-1}, t_{m-1,k-1}\bigr), \qquad m=1,\ldots,k,
\end{equation}
where we put $t_{0,k-1}:=+\infty$ and $t_{k,k-1}:=-\infty$ (for $k=2$ there are only two unbounded intervals $L_1=(0,+\infty)$ and $L_2=(-\infty,0)$).
So, for each $j$ there is exactly one root of the $j$-th relation \eqref{spefin} on each interval $L_{m,k}$. So all the roots of \eqref{spefin}
$\{\l_{j,m}=\l_{j,m}(n,k)\}_{j=1}^n$ (the eigenvalues of $\gg_{n,k}$) are simple and can be arranged in the groups
\begin{equation}\label{group}
\{\l_{1,m}>\l_{2,m}>\ldots>\l_{n,m}\}\subset L_{m,k}, \qquad m=1,\ldots,k.
\end{equation}

\smallskip

(ii). The eigenvalues of $\gg_{n,k}$ on the interval $(2,\infty)$ certainly belong to the first group $\{\l_{j,1}\}$. Choose
$p=p_{n,k}$ from the conditions
$$ 2\cos\frac{p\pi}{n+1}>\frac{k+1}{k}=v_k(2)>2\cos\frac{(p+1)\pi}{n+1}\,, $$
or
\begin{equation}\label{numeig}
p<\frac{\o_k}{\pi}\,(n+1)<p+1, \qquad \o_k=\arccos\frac{k+1}{2k},
\end{equation}
(note that, by Corollary \ref{cor1}, the number $\o_k/\pi$ is irrational, and so both inequalities in \eqref{numeig} are strict). Hence,
$\l_{p,1}>2>\l_{p+1,1}$, and \eqref{numeig1} follows.

\smallskip

(iii). By \eqref{spefin}, the maximal eigenvalue $\l_{1,1}=\l_{1,1}(n,k)$ satisfies
$$ v_k(\l_{1,1})=2\cos\frac{\pi}{n+1}<2. $$
Denote by $\l_1=\l_1(k)$ the maximal root of the equation $v_k(\cdot)=2$. Clearly, $\l_{1,1}<\l_1$, and also
$$ v_k(2)=\frac{V_k(2)}{V_{k-1}(2)}=\frac{k+1}{k}<2 $$
implies $\l_1>2$. Take $y_1=y_1(k)>0$ from the relation $\l_1=2\cosh y_1$, so, by \eqref{chedef},
\begin{equation}\label{sfunc}
2=\frac{U_k(\cosh y_1)}{U_{k-1}(\cosh y_1)}=\frac{\sinh (k+1)y_1}{\sinh ky_1}=:s_k(y_1).
\end{equation}
It is easy to check that $s_k(t)>e^t$ for all $t>0$, and so
\begin{equation}\label{spebou}
y_1(k)<\log 2, \qquad \l_{1,1}(n,k)<\l_1(k)<5/2, \qquad \forall n,k\ge2,
\end{equation}
as claimed.

The thorough analysis provides the announced extremal property of the interval $[-5/2, 5/2]$. Indeed, let $\s(\gg_{n,k})\subset [-a,a]$.
It follows from the construction that $\{\l_{1,1}(\cdot,k)\}$ is the monotone increasing sequence, and
\begin{equation}\label{spebou1}
\lim_{n\to\infty} \l_{1,1}(n,k)=\l_1(k), \qquad k=2,3,\ldots,
\end{equation}
so $\l_1(k)\in [-a,a]$, i.e., $\l_1(k)\le a$ for all $k$.

Next, by elementary calculus, it is not hard to see that each function $s_k$ \eqref{sfunc} is monotone increasing, whereas the sequence $\{s_k\}$
is monotone decreasing, that is, $s_{k+1}<s_k$, and so the sequence $\{y_1(k)\}$ is monotone increasing. On the other hand, $y_1(k)<\log2$, which means
that the limits exist
\begin{equation}\label{spebou2}
y_0:=\lim_{k\to\infty}y_1(k)\le \log2, \quad \l_0:=2\cosh y_0=\lim_{k\to\infty} \l_1(k)\le\frac52.
\end{equation}
Now, let us write the definition
$$ 2=s_k(y_1(k))=\frac{\sinh (k+1)y_1(k)}{\sinh ky_1(k)}=\frac{e^{y_1(k)}-e^{-(2k+1)y_1(k)}}{1-e^{-2k\,y_1(k)}} $$
and tend $k\to\infty$. We see that actually $y_0=\log 2$. But $2\cosh y_1(k)\le a$ implies, as $k\to\infty$,
\begin{equation}\label{spebou3}
\l_0=2\cosh y_0=\frac52\le a,
\end{equation}
as claimed. The proof is complete.
\end{proof}

Note that formula \eqref{charpol} holds (for obvious reasons) for $k=1$ and $n=1$, since
$$ P(\l,\gg_{n,1})=P(\l,\bp_n), \qquad P(\l,\gg_{1,k})=P(\l,\bp_k). $$

\begin{remark}
Since the maximal valency of the vertices of $\gg_{n,k}$ equals $3$, there is an obvious inclusion for the spectrum, which is a consequence
of Gershgorin's theorem, see \cite[Corollary 6.1.3]{hojohn},
$$ \s(\gg_{n,k})\subset [-3,3], \qquad n,k\ge2. $$
\end{remark}

\begin{example}\label{ex1}
Consider the comb graph $\gg_{n,2}$, see Fig.4. Its characteristic polynomial is
$$ P(\l,\gg_{n,2})=V_n\bigl(v_2(\l)\bigr)\,V_1^n(\l)=\l^n\,V_n\Bigl(\l-\frac1{\l}\Bigr). $$
The calculation of the bounds for $\s(\gg_{n,2})$ can be made more precise than for the general case (iii) above.

Indeed, the maximal eigenvalue $\l_{1,1}(n,2)$ is the maximal root of the quadratic equation
$$ \l-\frac1{\l}=2\cos\frac{\pi}{n+1}\,, \quad \l_{1,1}(n,2)=\cos\frac{\pi}{n+1}+\sqrt{\cos^2\frac{\pi}{n+1}+1}. $$
We see that $\l_{1,1}(n,2)<\l_1(2)=1+\sqrt2$ for all $n\ge2$, and so
\begin{equation}\label{inclex}
\s(\gg_{n,2})\subset [-1-\sqrt2, 1+\sqrt2]\subset \Bigl[-\frac52, \frac52\Bigr], \qquad n\ge2.
\end{equation}
In fact,
$$\l_{1,1}(2,2)<\l_{1,1}(3,2)<2<\l_{1,1}(4,2) \ \ \Rightarrow \ \ \s(\gg_{n,2})\subset(-2,2), \quad n=2,3, $$
cf. \cite[Theorem 3.1.3]{BrHae} with $\gg_{2,2}=\bp_4$, $\gg_{3,2}=E_6$, and the maximal eigenvalue (the index of the graph)
is bigger than $2$ for $n\ge4$.


\begin{picture}(360, 100)
\multiput(60,20) (30,0) {3} {\circle* {4}}
\multiput(140,20) (10,0) {3} {\circle* {2}}
\multiput(180,20) (30,0) {2} {\circle* {4}}
\multiput(60,50) (30,0) {3} {\circle* {4}}
\multiput(180,50) (30,0) {2} {\circle* {4}}
\multiput(62,20) (30,0) {2} {\line(1, 0) {26}}
\multiput(182,20) (30,0) {1} {\line(1, 0) {26}}
\multiput(122,20) (15,0) {1} {\line(1, 0) {11}}
\multiput(168,20) (15,0) {1} {\line(1, 0) {11}}
\put(250,24){\Large{$\Gamma_{n,2}$}}
\multiput(60,22) (30,0) {1} {\line(0, 1) {26}}
\multiput(90,22) (30,0) {1} {\line(0, 1) {26}}
\multiput(120,22) (30,0) {1} {\line(0, 1) {26}}
\multiput(180,22) (30,0) {1} {\line(0, 1) {26}}
\multiput(210,22) (30,0) {1} {\line(0, 1) {26}}
\put(58,8) {$1$}
\put(88,8) {$2$}
\put(118,8) {$3$}
\put(170,8) {$n-1$}
\put(208,8) {$n$}
\put(32,48) {$n+1$}
\put(62,48) {$n+2$}
\put(214,48) {$2n$}

\end{picture}

\begin{center}\image \end{center}

\end{example}

\bigskip

\section{Comb graphs with tail}
\label{s3}

We study here the coupling $H_{n,k}=\gg_{n,k}+\bp_\infty$ and prove Theorem \ref{mainthe}.

\begin{picture}(360, 200)
\multiput(60,40) (30,0) {3} {\circle* {4}}
\multiput(140,40) (10,0) {3} {\circle* {2}}
\multiput(180,40) (30,0) {2} {\circle* {4}}
\multiput(60,70) (30,0) {3} {\circle* {4}}
\multiput(180,70) (30,0) {2} {\circle* {4}}
\multiput(60,130) (30,0) {3} {\circle* {4}}
\multiput(180,130) (30,0) {2} {\circle* {4}}
\multiput(60,160) (30,0) {3} {\circle* {4}}
\multiput(180,160) (30,0) {2} {\circle* {4}}
\multiput(60,90) (0,10) {3} {\circle* {2}}
\multiput(90,90) (0,10) {3} {\circle* {2}}
\multiput(120,90) (0,10) {3} {\circle* {2}}
\multiput(180,90) (0,10) {3} {\circle* {2}}
\multiput(210,90) (0,10) {3} {\circle* {2}}
\multiput(62,40) (30,0) {2} {\line(1, 0) {26}}
\multiput(182,40) (30,0) {1} {\line(1, 0) {26}}
\multiput(60,42) (30,0) {1} {\line(0, 1) {26}}
\multiput(60,132) (30,0) {1} {\line(0, 1) {26}}
\multiput(90,42) (30,0) {1} {\line(0, 1) {26}}
\multiput(90,132) (30,0) {1} {\line(0, 1) {26}}
\multiput(120,42) (30,0) {1} {\line(0, 1) {26}}
\multiput(120,132) (30,0) {1} {\line(0, 1) {26}}
\multiput(180,42) (30,0) {1} {\line(0, 1) {26}}
\multiput(180,132) (30,0) {1} {\line(0, 1) {26}}
\multiput(210,42) (30,0) {1} {\line(0, 1) {26}}
\multiput(210,132) (30,0) {1} {\line(0, 1) {26}}
\put(210,28) {$v$}

\if{\put(60,28) {$1$}
\put(90,28) {$2$}
\put(120,28) {$3$}
\put(172,28) {$n-1$}
\put(210,28) {$n$}
\put(32,68) {$n+1$}
\put(64,68) {$n+2$}
\put(214,68) {$2n$}
\put(2,158) {$(k-1)n+1$}
\put(214,158) {$kn$}
\put(224,28) {$kn+1$}
\put(260,28) {$kn+2$}}\fi

\multiput(240,40) (30,0) {2} {\circle* {4}}
\multiput(290,40) (10,0) {3} {\circle* {2}}
\multiput(212,40) (30,0) {2} {\line(1, 0) {26}}

\put(300,120){\Large{$H_{n,k}$}}
\end{picture}

\begin{center}\image \end{center}

\bigskip

The adjacency matrix $A(H_{n,k})$ takes the form
\begin{equation}\label{adjtail}
A(H_{n,k})=A(\gg_{n,k})\bigoplus A(\bp_\infty)+\Delta, \qquad {\rm rank}\,\Delta=2,
\end{equation}
so the finiteness of the discrete spectrum follows.

Our argument relies heavily upon the result from \cite[Theorem 1 and relation (1.6)]{Gol17}. It states that $\l\in\s_d(H_{n,k})$ if and only if
\begin{equation}\label{maineq}
P(\l,\gg_{n,k})-x\,P(\l,\gg_{n,k}\backslash\{v\})=0, \qquad \l=x+\frac1x, \quad x\in(-1,1).
\end{equation}
It is clear that $\gg_{n,k}\backslash\{v\}=\gg_{n-1,k}\cup\bp_{k-1}$, the disjoint union, and so
$$ P(\l,\gg_{n,k}\backslash\{v\})=P(\l,\gg_{n-1,k})\cdot P(\l,\bp_{k-1}). $$
We have, by \eqref{charpol} and \eqref{symdet},
$$ V_n\bigl(v_k(\l)\bigr)\cdot V_{k-1}^n(\l)-x\,V_{n-1}\bigl(v_k(\l)\bigr)\,V_{k-1}^{n-1}(\l)\cdot V_{k-1}(\l)=0. $$
Since $V_{k-1}\not=0$ on $\br\backslash[-2,2]$, we can cancel $V_{k-1}^n$ out to obtain
\begin{equation}\label{maineq1}
R_{n,k}(x):=V_n\bigl(v_k(\l)\bigr)-x\,V_{n-1}\bigl(v_k(\l)\bigr)=0, \qquad \l=x+\frac1x\,.
\end{equation}
Since $R_{n,k}(-\l)=(-1)^n R_{n,k}(\l)$, the discrete spectrum is symmetric with respect to the origin.
We restrict ourselves with its positive part $\s_d^+(H_{n,k})$.

Write \eqref{maineq1} as
\begin{equation}\label{maineq2}
v_{n,k}(\l):=\frac{V_{n-1}\bigl(v_k(\l)\bigr)}{V_{n}\bigl(v_k(\l)\bigr)}=\frac1x=\frac{\l+\sqrt{\l^2-4}}2=:v(\l), \quad \l\ge2.
\end{equation}
The function $v_k$ \eqref{charpol} is regular on $(2,\infty)$ and grows monotonically from $k+1/k$ to $+\infty$. The zeros of the denominator
in \eqref{maineq2} are exactly the eigenvalues of $\gg_{n,k}$, see \eqref{spefin}. So, the only eigenvalues $\l_{j,1}(n,k)$ in the first group can
arise as the singularities of the rational function on the left side \eqref{maineq2} on the half-line $(2,\infty)$
$$ \{\l_{1,1}>\l_{2,1}>\ldots>\l_{n,1}\}\subset L_1=\bigl(2\cos\frac{\pi}{k}, \infty\bigr). $$

Take $p=p_{n,k}$ from \eqref{numeig}, so that $\l_{p,1}>2>\l_{p+1,1}$. The function $v_{n,k}$ is monotone decreasing from $+\infty$ to $-\infty$ on
each interval $(\l_{j+1,1},\l_{j,1})$, $j=1,\ldots,p-1$, and from $+\infty$ to $0$ on $(\l_{1,1},\infty)$, whereas $v$ is monotone increasing from
$1$ to $+\infty$ on $(2,\infty)$. Hence, there is exactly one root of \eqref{maineq2} on each such interval, that is, $p$ roots altogether. The existence
of the root on $(2,\l_{p,1})$ depends on the value
\begin{equation}\label{adzero}
a_{n,k}:=v_{n,k}(2)=\frac{U_{n-1}\bigl(\frac{k+1}{2k}\bigr)}{U_{n}\bigl(\frac{k+1}{2k}\bigr)}=\frac{\sin n\o_k}{\sin (n+1)\o_k}\,.
\end{equation}
Specifically, there is one root as soon as $a_{n,k}>1$, and there is no such root otherwise (note that, by Corollary \ref{cor1}, $a_{n,k}\not=1$).
Thereby, relation \eqref{numpds} is established.

Denote by $\nu_j=\nu_j(n,k)$ the {\it different} eigenvalues in $\s_d^+(H_{n,k})$, arranged in non-increasing order (we will discuss the
multiplicity of the eigenvalues later on). So, $\nu_1=\nu_1(n,k)$ is the maximal such eigenvalue, that is, the root of \eqref{maineq2} on $(\l_{1,1},\infty)$.
Clearly, $\nu_1>\l_{1,1}$. As in the proof of Theorem \ref{fincomb}, (iii), denote by $\l_1=\l_1(k)$ the maximal root of the equation $v_k(\cdot)=2$. Let us
compare the values of both sides in \eqref{maineq2} at $\l=\l_1$:
$$ v_{n,k}(\l_1)=\frac{V_{n-1}(2)}{V_{n}(2)}=\frac{n}{n+1}<1<v(\l_1), $$
so $\nu_1<\l_1$. But, as we have already shown in \eqref{spebou}, $\l_1<5/2$, which means that $\s_d^+(H_{n,k})\subset [2,5/2]$, as needed.

Moreover, assume that $\s_d^+(H_{n,k})\subset [2,a]$ for all $n,k\ge2$. Then, by \eqref{spebou1},
\begin{equation}\label{majeig}
\l_{1,1}(n,k)<\nu_1(n,k)<\l_1(k) \ \ \Rightarrow \ \ \lim_{n\to\infty} \nu_1(n,k)=\l_1(k)\le a.
\end{equation}
The rest is the same as in the proof of Theorem \ref{fincomb}, (iii).

The second inclusion in \eqref{inclspe} and the extremal property of the corresponding interval follows from \eqref{majeig} with $k=2$, since
$\l_1(2)=1+\sqrt2$, see \eqref{inclex} in Example \ref{ex1}.

Finally, let us turn to the multiplicity issue. Put $A:=A(\gg_{n,k})\bigoplus A(\bp_\infty)$, $B:=H_{n,k}$. In view of \eqref{adjtail}, $B$ is the rank $2$
perturbation of $A$. Take the interval $I:=(2,3)$, then both $A$ and $B$ have a pure discrete spectrum on $I$. Furthermore, $\s_d(A)$ on $I$ consists of $p=p_{n,k}$ simple
eigenvalues, so $\pi_A(I)=p$. Hence, by Theorem \ref{finrank}, the spectral multiplicity of $B$ on $I$ is at most $p+2$. But, we know that the number of
{\it different} eigenvalues of $B$ on $I$ is $p$ or $p+1$, and there are no common eigenvalues for $A$ and $B$. Then there are at most two eigenvalues of $B$,
having multiplicity $2$ (there are no eigenvalues of $B$ of multiplicity $3$). The argument for the negative part of $\s_d(B)$ is identical.

The proof is complete.

\begin{remark}
Let $\gg$ be a finite graph of order $m$, $H=\gg+\bp_\infty$. According to \cite[Theorem 1.2]{Gol16}, the adjacency operator $A(H)$ is unitarily equivalent
to the orthogonal sum
$$ U^{-1}A(H)U=F(H)\bigoplus J(H), $$
where $F(H)$ is an operator on a finite dimensional subspace of dimension at most $m-1$, $J(H)$ is a Jacobi operator of finite rank at most $m$,
generated by a Jacobi (tridiagonal) matrix
\begin{equation*}
J=J(\{b_j\}, \{a_j\})=
\begin{bmatrix}
 b_1 & a_1 & & \\
 a_1 & b_2 & a_2 & \\
     & a_2 & b_3 & \ddots & \\
     &     & \ddots & \ddots
\end{bmatrix}, \quad b_j\in\br, \quad a_j>0,
\end{equation*}
in an appropriate orthonormal basis. $J$ is said to have a finite rank if
$$ d_j:=|b_j|+|a_j-1|=0, \qquad j\ge j_0, $$
and the minimal such $j_0$ is called the rank of $J$. $F(H)$ and $J(H)$ are referred to as the finite and Jacobi components of $H$, respectively.

The structure of $\s(J(H))$ is quite transparent: the absolutely continuous spectrum on $[-2,2]$ along with a finite number of simple eigenvalues
off $[-2,2]$, which come from the roots of a certain algebraic equation (the Jost equation, see \cite[Section 5]{Gol16}). The eigenvalues of $F(H)$
can appear anywhere on the real line including $[-2,2]$. It is not clear how to find this later set (the ``hidden spectrum'') without computing the
finite component $F(H)$ explicitly, the problem, which seems to be hard enough. For example, if $H=H_{3,2}$, the finite component is missing (no
hidden spectrum), and the Jacobi component is given by $J(\{b_j\}, \{a_j\})$ with
$$ b_j\equiv0; \quad a_1=\frac1{\sqrt3}\,, \ a_2=\frac1{\sqrt6}\,, \ a_3=\sqrt{\frac32}\,, \ a_4=1, \ a_5=\sqrt2, \ a_6=\ldots=1. $$

Some information on the hidden spectrum can be gathered from \cite[Theorem 4.3]{abcre}, which claims that $A(H_{n,k})$ has a trivial kernel, i.e.,
$0$ is not an eigenvalue of $A(H_{n,k})$.
\end{remark}

In the table below we display the number of the positive eigenvalues for the graph $H_{n,k}$ for $n=2,3,\ldots,20$ and $k=2,3,\ldots,6$ obtained
by using the computer experiment.

\begin{table}[h]

\begin{tabular}{|c|c c c c c c c c c c c c c c c c c c c|}
\hline
$(k,n)$
&2&3&4&5&6&7&8&9&10&11&12&13&14&15&16&17&18&19&20 \\
\hline
2&1&1&1&1&1&2&2&2&2&3&3&3&3&4&4&4&4&4&5\\
3&1&1&1&1&2&2&2&3&3&3&3&4&4&4&4&5&5&5&5\\
4&1&1&1&2&2&2&2&3&3&3&4&4&4&4&5&5&5&6&6\\
5&1&1&1&2&2&2&3&3&3&3&4&4&4&5&5&5&5&6&6\\
6&1&1&1&2&2&2&3&3&3&3&4&4&4&5&5&5&6&6&6\\
\hline

\end{tabular}

\vskip 1cm

\caption{Cardinality of positive discrete spectrum $\sigma_d^+(H_{n,k})$}\label{tab}
\end{table}

\section{Appendix}

\subsection{Transformation of block matrices}
We collect here some information about operations on block matrices, see, e.g., \cite[Sections 0.7, 0.8]{hojohn}.
We denote by $\cm_m$ the set of all square matrices of order $m$.

Let
\begin{equation*}
\dd=\begin{bmatrix}
A & B \\
C & D
\end{bmatrix}
\end{equation*}
be a block matrix, $A$ and $D$ the square matrices (in general, of the different order). Then for the determinant $|\dd|=\det(\dd)$
the following relations hold
\begin{equation}\label{blma1}
\begin{split}
|\dd| &= |A|\,|D-CA^{-1}B|, \qquad |A|\not=0, \\
|\dd| &= |A-BD^{-1}C|\,|D|, \qquad |D|\not=0.
\end{split}
\end{equation}
\if{If all matrices $A,B,C,D\in\cm_m$, one can multiply out
\begin{equation*}
\begin{split}
|\dd| &= |AD-ACA^{-1}B|, \qquad |A|\not=0, \\
|\dd| &= |AD-BD^{-1}CD|, \qquad |D|\not=0.
\end{split}
\end{equation*}
In particular, if $AC=CA$, $(CD=DC)$, then
$$ |\dd|=|AD-CB|, \qquad (|\dd|=|AD-BC|). $$
}\fi
The matrices $D-CA^{-1}B$ and $A-BD^{-1}C$ are usually referred to as the {\it Schur complements}.

\smallskip

Let $A=\|A_{ij}\|_{i,j=1}^m$, $A_{ij}\in\cm_n$, be a block matrix. We call $A$ a {\it block matrix of scalar type} if each
block is a scalar matrix, i.e., $A_{ij}=a_{ij}I_n$, $I_n$ is the identity matrix of order $n$, $a_{ij}\in\bc$. We call the matrix
$a=\|a_{ij}\|_{i,j=1}^m\in\cm_m$ a symbol of $A$, and $A$ the inflation of $a$, $A=I(a)$.

The operations on block matrices of scalar type are easily reduced to the corresponding operations on their symbols
\begin{equation}\label{bmst}
\l A+\mu B=I(\l a+\mu b), \quad A^{-1}=I(a^{-1}), \quad |I(a)|=|a|^n.
\end{equation}
If $\s(a)=\{\l_j\}_{j=1}^p$ be the set of all different eigenvalues of $a$, and the algebraic multiplicity of $\l_j$ equals $\kappa_j$, then
$\s(I(a))=\{\l_j\}_{j=1}^p$, with the algebraic multiplicity $n\kappa_j$.

\subsection{Rational values of the arccosine functions}
The arithmetic properties of trigonometric functions has been a recurring topic in the number theory. In particular,
all the rational values of the function $\pi^{-1}\,\arccos t$, $t\in\bq$, the set of the rational numbers, can be determined
explicitly. This is a special case of a result of D. Lehmer \cite{leh}. For completeness sake we provide an elementary argument,
borrowed from \cite{jah}.

\begin{proposition}\label{ircos}
The only rational values of the function $\cos(\pi r)$, $r\in\bq$ are $0$, $\pm\frac12$, $\pm1$.
\end{proposition}
\begin{proof}
Put $g_j=g_j(r):=2\cos(2^j\,\pi r)$, $j=0,1,\ldots$. The addition formula for cosine implies
$$ g_{j+1}=g_j^2-2. $$
By the hypothesis, $g_0\in\bq$, and so are $g_j$ for all $j\ge0$. Write the irreducible fraction $g_0=a_0/b_0$. Then
$$ g_1=\frac{a_1}{b_1}=\frac{a_0^2-2b_0^2}{b_0^2}\,. $$
Clearly, the latter fraction is irreducible, since the assumption ${\rm gcd}(a_0,b_0)=~\!1$ implies ${\rm gcd}(a_0^2,b_0^2)=1$, and so ${\rm gcd}(a_0^2-2b_0^2,b_0^2)=1$.
Hence, the irreducible representation for $g_j$ is
$$ g_j=\frac{a_j}{b_j}\,, \qquad a_j=a_{j-1}^2-2, \quad b_j=b_0^{2^j}. $$
If $|b_0|\ge2$, the denominators $b_j$ tend to infinity.

On the other hand, as it follows from periodicity of cosine, and $r\in\bq$, the sequence $\{g_j\}_{j\ge0}$ takes the {\it finite} number of values, that is
inconsistent with $\b_j\to\infty$. So, $b_0=\pm1$ and $a_0=0,\,\pm1,\,\pm2$, as claimed.
\end{proof}

\begin{corollary}\label{cor1}
{\rm (i)}. The numbers
$$ \frac{\o_k}{\pi}=\frac1{\pi}\,\arccos\frac{k+1}{2k}\,, \qquad k=2,3,\ldots $$
are irrational.

{\rm (ii)}. Let $a_{n,k}$ be the numbers defined in $\eqref{numpds}$. Then $a_{n,k}\not=1$ for all $n,k=2,3,\ldots$.
\end{corollary}
\begin{proof}
The first statement is an immediate consequence of Proposition \ref{ircos}. As for the second one, assume that
$$ a_{n,k}=\frac{\sin n\o_k}{\sin(n+1)\o_k}=1, \qquad \sin(n+1)\o_k-\sin n\o_k=0, $$
or
$$ 2\sin\frac{\o_k}2\,\cos\Bigl(n+\frac12\Bigr)\,\o_k=0. $$
The first factor above is nonzero, so $\cos\Bigl(n+\frac12\Bigr)\,\o_k=0$, that is
$$ \frac{\o_k}{\pi}=\frac{2m+1}{2n+1}\,, $$
which contradicts (i).
\end{proof}

\subsection{Finite rank perturbations of self-adjoint operators}
Given a self-adjoint operator $T=T^*$ on the Hilbert space $\ch$, denote by $E_T$ its resolution of the identity.
For an interval $I\subset\br$ the operator $E_T(I)$ is the projection on the spectral subspace $E_T(I)\,\ch$. The dimension $\pi_T(I)=\dim E_T(I)\,\ch$
of this subspace is called a spectral multiplicity of $T$ on $I$.

We say that $T$ has a discrete spectrum on $I$ if $\s(T)\cap I\subset\s_d(T)$. In this case $I$ contains only isolated eigenvalues
of finite multiplicity, and the spectral multiplicity equals the sum of multiplicities of all eigenvalues in $I$.

The following result is well known, see, e.g., \cite[Theorem 9.3.3]{biso87}.

\begin{theorem}\label{finrank}
Let $A$, $B$ be self-adjoint operators on the Hilbert space $\ch$, so that
$$ B=A+\dd, \qquad {\rm rank}\,\dd=r<\infty. $$
If the spectrum of $A$ is discrete on a bounded interval $I$, then so is the spectrum of $B$, and $|\pi_B(I)-\pi_A(I)|\le r$.
\end{theorem}

\end{document}